\documentclass[a4paper]{amsart}

\usepackage[T1]{fontenc}
\usepackage{amsmath, amssymb, amsthm, mathrsfs, mathtools, marvosym, graphicx}
\usepackage[backref=page]{hyperref}
\usepackage{enumitem, xspace, ifthen, comment}
\usepackage[all]{xy}
\usepackage[dvipsnames]{xcolor}
\usepackage[nameinlink, capitalize]{cleveref}
\usepackage{tikz}
\usetikzlibrary{trees, calc}

\setlist[enumerate]{
  label=(\thethm.\arabic*),
  before={\setcounter{enumi}{\value{equation}}},
  after={\setcounter{equation}{\value{enumi}}},
  itemsep=1ex
}

\setlist[itemize]{
  leftmargin=*,
  itemsep=1ex,
  label=$\circ$
}

\newcommand{\mypagesize}{
  \addtolength{\textwidth}{4pt}
  \addtolength{\textheight}{27pt}
  \calclayout
}

\mypagesize

\newtheorem*{thm-plain}{Theorem}
\newtheorem{thm}{Theorem}[section]
\newtheorem{lem}[thm]{Lemma}
\newtheorem{prp}[thm]{Proposition}
\newtheorem{cor}[thm]{Corollary}
\newtheorem{fct}[thm]{Fact}

\numberwithin{equation}{thm}

\theoremstyle{definition}
\newtheorem{dfn}[thm]{Definition}

\newtheorem*{dfn-plain}{Definition}

\theoremstyle{remark}
\newtheorem{clm}[thm]{Claim}

\newtheorem{awlog}[thm]{Additional Assumption}

\newtheorem{exm}[thm]{Example}
\newtheorem*{rem-plain}{Remark}

\newcommand{\inv}{^{-1}}
\newcommand{\from}{\colon}
\newcommand{\gdw}{\ensuremath{\Leftrightarrow}}
\newcommand{\imp}{\ensuremath{\Rightarrow}}
\newcommand{\lto}{\longrightarrow}

\newcommand{\x}{\times}
\newcommand{\inj}{\hookrightarrow}

\newcommand{\bij}{\overset\sim\lto}
\newcommand{\isom}{\cong}
\newcommand{\defn}{\coloneqq}
\newcommand{\ndef}{\eqqcolon}

\newcommand{\wt}{\widetilde}

\renewcommand{\d}{\mathrm d}

\newcommand{\dual}{^{\smash{\scalebox{.7}[1.4]{\rotatebox{90}{\textup\guilsinglleft}}}}}
\newcommand{\ddual}{^{\smash{\scalebox{.7}[1.4]{\rotatebox{90}{\textup\guilsinglleft} \hspace{-.5em} \rotatebox{90}{\textup\guilsinglleft}}}}}

\newcommand{\factor}[2]{\left. \raise 2pt\hbox{$#1$} \right/\hskip -2pt \raise -2pt\hbox{$#2$}}

\newdir{ ir}{{}*!/-5pt/@^{(}} 
\newdir{ il}{{}*!/-5pt/@_{(}} 

\newcommand{\set}[1]{\left\{ #1 \right\}}

\def\rd#1.{\lfloor{#1}\rfloor}
\def\rp#1.{\lceil{#1}\rceil}
\def\tw#1.{\langle{#1}\rangle}

\renewcommand{\O}[1]{\mathscr{O}_{#1}}
\newcommand{\Omegap}[2]{\Omega_{#1}^{#2}}
\newcommand{\Omegar}[2]{\Omega_{#1}^{[#2]}}
\newcommand{\T}[1]{\mathscr{T}_{#1}}
\newcommand{\can}[1]{\omega_{#1}}

\newcommand{\Reg}[1]{{#1}_{\mathrm{reg}}}
\newcommand{\Sing}[1]{{#1}_{\mathrm{sg}}}
\newcommand{\Red}[1]{{#1}_{\mathrm{red}}}

\newcommand{\cc}[2]{\mathrm{c}_{#1}(#2)}

\def\Hnought#1.#2.{\mathit{\Gamma} \!\left( #1, #2 \right)}
\def\HH#1.#2.#3.{\mathrm{H}^{#1} \!\left( #2, #3 \right)}
\def\euler#1.#2.{\chi \!\left( #1, #2 \right)}
\def\HHbig#1.#2.#3.{\mathrm{H}^{#1} \!\big( #2, #3 \big)}
\def\hh#1.#2.#3.{h^{#1} \!\left( #2, #3 \right)}
\def\RR#1.#2.#3.{R^{#1} #2_* #3}
\def\HHc#1.#2.#3.{\mathrm{H}_{\mathrm{c}}^{#1} \!\left( #2, #3 \right)}
\def\Hh#1.#2.#3.{\mathrm{H}_{#1} \!\left( #2, #3 \right)}
\def\Hom#1.#2.{\mathrm{Hom} \!\left( #1, #2 \right)}
\def\sHom#1.#2.{\mathscr{H}\!om \!\left( #1, #2 \right)}
\def\Ext#1.#2.#3.{\mathrm{Ext}^{#1} \!\left( #2, #3 \right)}
\def\sExt#1.#2.#3.{\mathscr{E}\!xt^{#1} \!\left( #2, #3 \right)}

\newcommand{\PP}[1]{\mathbb P^{#1}}

\newcommand{\kahler}{K{\"{a}}hler\xspace}

\DeclareMathOperator{\Picn}{Pic^\circ}

\DeclareMathOperator{\Exc}{Exc}

\newcommand{\pg}[2]{p_g(#1, #2)}
\newcommand{\germ}[2]{\left( #2, #1 \right)} 

\renewcommand{\theta}{\vartheta}
\renewcommand{\phi}{\varphi}

\newcommand{\N}{\ensuremath{\mathbb N}}
\newcommand{\Z}{\ensuremath{\mathbb Z}}
\newcommand{\Q}{\ensuremath{\mathbb Q}}
\newcommand{\R}{\ensuremath{\mathbb R}}
\newcommand{\C}{\ensuremath{\mathbb C}}

\renewcommand{\P}{\ensuremath{\mathbb P}}

\renewcommand{\frm}{\mathfrak m}

 \newcommand{\sE}{\mathscr E} \newcommand{\sF}{\mathscr F}
\newcommand{\sG}{\mathscr G}  
 \newcommand{\sK}{\mathscr K} \newcommand{\sL}{\mathscr L}

 \newcommand{\cQ}{\mathcal Q}

\definecolor{forrest}{RGB}{81,133,49}
\definecolor{mydarkblue}{RGB}{10,92,153}

\newcommand{\PreprintAndPublication}[2]{#1}

\title{The Lipman--Zariski conjecture in low genus}

\dedicatory{Cad{\^{e}} vi{\'{o}}la? Cad{\^{e}} meu bem?}

\author{Patrick Graf}
\address{Lehrstuhl f\"ur Mathematik I, Universit\"at Bayreuth, 95440 Bayreuth, Germany}
\email{\href{mailto:patrick.graf@uni-bayreuth.de}{patrick.graf@uni-bayreuth.de}}
\urladdr{\href{http://www.pgraf.uni-bayreuth.de/en/}{www.graficland.uni-bayreuth.de}}

\date{May 6, 2021}

\hypersetup{
  pdfauthor={Patrick Graf},
  pdftitle={},
  pdfkeywords={},
  pdfstartview={Fit},
  pdfpagelayout={OneColumn},
  pdfpagemode={UseNone},
  linktoc=all,
  breaklinks,
  linkcolor=[RGB]{0 0 96},
  citecolor=[RGB]{96 0 0},
  urlcolor=[RGB]{0 96 0},
  colorlinks}

\begin{document}

\begin{abstract}
We prove the Lipman--Zariski conjecture for complex surface singularities of genus one, and also for those of genus two whose link is not a rational homology sphere.
As an application, we characterize complex $2$-tori as the only normal compact complex surfaces whose smooth locus has trivial tangent bundle.
We also deduce that all complex-projective surfaces with locally free and generically nef tangent sheaf are smooth, and we classify them.
\end{abstract}

\maketitle


\section{Introduction}

The Lipman--Zariski conjecture asserts that a complex algebraic variety $X$ with locally free tangent sheaf $\T X$ is necessarily smooth.
Here $\T X = \sHom \Omegap X1.\O X.$ is the dual of the sheaf of \kahler differentials.
It is known that such an $X$ is at least normal~\cite[Thm.~3]{Lip65}.
Moreover, if the conjecture fails then there is a counterexample with isolated singularities~\cite[Sec.~8, p.~519]{Becker78}.
Finally, the conjecture holds if the singular locus of $X$ has codimension $\ge 3$~\cite[Corollary]{Flenner88}.
Taken together, these results show that \emph{it suffices to consider the case of normal surface singularities.}

A natural approach to the Lipman--Zariski conjecture is to study it under additional assumptions on the singularities of $X$.
A well-known and interesting class of singularities is given by the class of rational singularities.
For these, the conjecture is already known in any dimension, since rational Gorenstein singularities are canonical and the conjecture is true even more generally for log canonical singularities~\cite[Thm.~1.1]{Dru13}, \cite[Cor.~1.3]{GK13}.
In this paper we deal with surface singularities that are ``not too far'' from being rational, in the sense that their (geometric) genus is low.
Recall that the genus $\pg X0$ of a surface singularity $\germ0X$ is the dimension of $\RR1.f.\O Y.$, for a resolution $f \from Y \to X$ (see \cref{pg}).

\begin{thm}[Lipman--Zariski conjecture in low genus] \label{lowgenus lz}
Let $\germ0X$ be a normal complex surface singularity, i.e.~a germ of a two-dimensional normal complex space.
Assume that either
\begin{enumerate}
\item \label{genus1} $\pg X0 \le 1$, or
\item \label{genus2} $\pg X0 = 2$ and for some log resolution $f \from Y \to X$, the exceptional locus $E = \Exc(f)$ is \emph{not} a tree of rational curves.
\end{enumerate}
Then the Lipman--Zariski conjecture holds for $\germ0X$.
That is, if $\T X$ is free, then $\germ0X$ is smooth.
\end{thm}

\begin{rem-plain}
The condition on $E$ in \labelcref{genus2} does not depend on the choice of resolution, since it is equivalent to the link of $\germ0X$ not being a rational homology sphere\PreprintAndPublication{ (see \cref{ratl hom sphere})}{}.
\end{rem-plain}

\begin{rem-plain}
There exist Gorenstein surface singularities of genus $1$ which are not log canonical, hence to which~\cite{Dru13, GK13} do not apply.
On the other hand, somewhat surprisingly at first sight, by contracting a tree of rational curves one can obtain singularities of arbitrarily high genus.
Examples of this type are given in \cref{examples}.
\end{rem-plain}

\subsection*{Corollaries}

As an application, we study a global version of the Lipman--Zariski conjecture.
Namely, assume that $X$ is a compact complex surface with \emph{globally} free tangent sheaf.
\emph{Does it follow that $X$ is smooth?}
Note that we may equivalently assume $X$ to be a normal compact complex surface whose smooth locus $\Reg X$ has trivial tangent bundle.

Partial answers have been given by Ballico~\cite[Thm.~2]{Ballico06} and Biswas--Gurjar--Kolte~\cite[Thm.~1.2]{BiswasGurjarKolte14}.
Our main result enables us to settle this question completely, even under weaker assumptions.
Recall that a (reduced and connected) compact complex space $X$ is called \emph{almost homogeneous} if its automorphism group acts with a dense open orbit.
This is equivalent to the tangent sheaf $\T X$ being globally generated at some point.

\begin{cor}[Global LZ conjecture, I] \label{X almost hom}
Let $X$ be an almost homogeneous compact complex surface such that $\T X$ is locally free.
Then $X$ is smooth.
\end{cor}

An immediate consequence is

\begin{cor}[Global LZ conjecture, II] \label{TX trivial}
Let $X$ be a compact complex surface such that $\T X \isom \O X^{\oplus 2}$.
Then~$X$ is a complex $2$-torus.
\end{cor}

\begin{rem-plain}
In the above corollaries, we do not have to assume explicitly that $X$ is normal since this is automatic by~\cite[Thm.~3]{Lip65}.
\end{rem-plain}

\begin{rem-plain}
The almost homogeneous \emph{smooth} compact complex surfaces have been classified by Potters~\cite{PottersAlmostHomogeneous}.
Also, a compact \kahler \emph{manifold} (of arbitrary dimension) with trivial tangent bundle is necessarily a complex torus by~\cite[Cor.~2]{Wang54}.
This fails if the \kahler condition is dropped, the historically first example being the Iwasawa manifold.
It also fails in positive characteristic~\cite{MehtaSrinivas87}.
\end{rem-plain}

\begin{rem-plain}
After submission of this paper, the author learned from an anonymous referee that \cref{TX trivial} can also be obtained as a direct consequence of~\cite[Cor.~2]{OeljeklausRichthofer88}.
\end{rem-plain}

If $X$ is projective, we can weaken the assumptions on $\T X$ further.
Recall that a vector bundle $\sE$ on a normal projective variety $X$ of dimension $n$ is said to be \emph{generically nef (with respect to \emph{some} polarization)} if there exist ample line bundles $H_1, \dots, H_{n-1}$ on $X$ with the following property:
Let $C \subset X$ be a curve cut out by general elements of the linear system $|m_i H_i|$, for $m_i \gg 0$.
Then the restriction $\sE|_C$ is nef.
\emph{Generic ampleness} is defined similarly.

\begin{cor}[Global LZ conjecture, III] \label{TX gen nef}
Let $X$ be a complex-projective surface such that $\T X$ is locally free and generically nef.
Then $X$ is smooth.
More precisely, one of the following holds.
\begin{enumerate}
\item \label{P2} $X \isom \PP2$.
\item \label{ratl ruled} $X$ has a surjective birational morphism onto a rational ruled surface.
\item \label{ruled ell} $X$ has a surjective birational morphism onto a ruled surface over an elliptic curve $C$ such that all fibres of the map $X \to C$ are reduced.
\item \label{abelian} $X$ is an abelian or a bi-elliptic surface.
\item \label{K3} $X$ is a projective K3 surface or an Enriques surface.
\end{enumerate}
Conversely, for the surfaces in the above list, the tangent bundle is: \labelcref{P2} ample, \labelcref{ratl ruled} generically ample, \labelcref{abelian} nef, \labelcref{ruled ell} and \labelcref{K3} generically nef.
\end{cor}

Under the stronger assumption of generic ampleness, the first part of \cref{TX gen nef} has been proved by Ballico~\cite[Thm.~1]{Ballico06}.

\subsection*{Acknowledgements}

The proof of \cref{TX trivial} is due to Hannah Bergner (previously we relied on the Kodaira--Enriques classification of surfaces).
Two anonymous referees have made useful suggestions concerning the presentation of the paper.

\section{Notation and basic facts}

We work over the field of complex numbers \C.
The \emph{sheaf of \kahler differentials} of an algebraic variety or reduced complex space $X$ is denoted $\Omegap X1$.
The \emph{tangent sheaf}, its dual, is denoted $\T X \defn \sHom \Omegap X1.\O X.$.
If $Z \subset X$ is a closed subset, then $\T X(-\log Z) \subset \T X$ denotes the subsheaf of derivations stabilizing the ideal sheaf of $Z$ (geometrically, this means vector fields tangent to $Z$ at every point of $Z$).
If $X$ is normal, its \emph{canonical sheaf} (the sheaf of reflexive differential $n$-forms) is denoted by $\can X$.
More generally, for any $p \in \N$, the \emph{sheaf of reflexive differential $p$-forms} is defined to be the double dual of $\bigwedge^p \Omegap X1$.
We denote it by $\Omegar Xp \defn \left( \bigwedge^p \Omegap X1 \right) \ddual$, and it is isomorphic to $i_* \big( \Omegap{X^\circ}p \big)$, where $i \from X^\circ \inj X$ is the inclusion of the smooth locus.

\begin{dfn}[Resolutions]
A \emph{resolution of singularities} of an algebraic variety or reduced complex space $X$ is a proper birational/bimeromorphic morphism $f \from Y \to X$, where $Y$ is smooth.
\begin{enumerate}
\item We say that the resolution is \emph{projective} if $f$ is a projective morphism.
That is, $f$ factors as $Y \inj X \x \PP n \to X$, where the first map is a closed embedding and the second one is the projection.
\item A \emph{log resolution} is a resolution whose exceptional locus $E = \Exc(f)$ is a simple normal crossing divisor, i.e.~a normal crossing divisor with smooth components.
\item A resolution is said to be \emph{strong} if it is an isomorphism over the smooth locus of $X$.
\end{enumerate}
\end{dfn}

\begin{fct}[Functorial resolutions] \label{functorial res}
Let $X$ be a normal algebraic variety or complex space.
Then there exists a projective\footnote{If $X$ is a complex space, then projectivity of $f$ is only guaranteed over compact subsets of $X$.} strong log resolution $f \from Y \to X$, called the \emph{functorial resolution}, such that $f_* \T Y(-\log E)$ is reflexive.
This means that for any vector field $\xi \in \Hnought U.\T X.$, $U \subset X$ open, there is a unique vector field
\[ \wt\xi \in \Hnought f\inv(U).\T Y(-\log E). \]
which agrees with $\xi$ wherever $f$ is an isomorphism.
\end{fct}

\cref{functorial res} is proven in~\cite[Thms.~3.36 and~3.45]{Kol07}, but concerning the reflexivity of $f_* \T Y(-\log E)$ see also~\cite[Thm.~4.2]{GK13}.
If $X$ is a surface, the functorial resolution is also known as the \emph{minimal good resolution}.
Mapping $\xi \mapsto \wt\xi$ gives a sheaf map $\T X \bij f_* \T Y(-\log E)$, which by adjointness~\cite[Ch.~II, Sec.~5, p.~110]{Har77} can also be regarded as a map of sheaves on $Y$,
\[ f^* \from f^* \T X \lto \T Y(-\log E). \]
We will call both maps the \emph{pullback map on vector fields.}

\begin{dfn}[Geometric genus] \label{pg}
Let $\germ0X$ be a normal surface singularity, and let $f \from Y \to X$ be a resolution.
The \emph{(geometric) genus} $\pg X0$ is defined to be the dimension of the stalk $(\RR1.f.\O Y.)_0$.
Alternatively, choosing the representative $X$ of the germ $\germ0X$ to be Stein, $\pg X0 \defn \dim_\C \HH1.Y.\O Y.$.
This definition is independent of the choice of $f$.
\end{dfn}

\section{Proof of \cref{lowgenus lz}}

Let $\germ0X$ be a normal surface singularity and $f \from Y \to X$ a log resolution with reduced exceptional divisor $E \subset Y$.
Our proof relies on the following special case of a result by Steenbrink and van Straten, which in turn ultimately stems from the Steenbrink vanishing theorem~\cite[Thm.~2.b)]{Ste85}.

\begin{thm}[\protect{\cite[Cor.~1.4]{SvS85}}] \label{SvS}
The map
\[ \factor{\Omegar X1}{f_* \Omegap Y1} \xrightarrow{\quad\d\quad} \factor{\can X}{f_* \can Y(E)} \]
induced by the exterior derivative is injective. \qed
\end{thm}

\begin{lem}[Trees of rational curves] \label{ratl tree}
Let $C$ be a proper connected reduced curve with simple normal crossings.
The following are equivalent.
\begin{enumerate}
\item \label{C ratl tree} $C$ is a tree of rational curves, that is, every irreducible component of $C$ is isomorphic to $\PP1$ and the dual graph of $C$ does not contain any cycles.
\item \label{O_C van} $\HH1.C.\O C. = 0$.
\item \label{Q_C van} $\HH1.C.\Q. = 0$.
\end{enumerate}
\end{lem}

\begin{proof}
Write $C = \bigcup_{i=1}^n C_i$ for the decomposition into irreducible components and let $\nu \from C^\nu = \coprod_{i=1}^n C_i \to C$ be the normalization map.
By the long exact sequence associated to
\[ 0 \lto \O C \lto \nu_* \O{C^\nu} \lto \bigoplus_{P \in \Sing C} \underline \C_P \lto 0, \]
we deduce that
\begin{equation} \label{421}
\hh1.C.\O C. = \sum_{i=1}^n \hh1.C_i.\O{C_i}. + \underbrace{\#(\Sing C) - n + 1}_{\text{$\ge 0$, $C$ connected}}.
\end{equation}
Hence if $\hh1.C.\O C.$ is zero, then each $C_i$ is rational and $\#(\Sing C) = n - 1$, that is, $C$ is a tree.
Conversely, if $C$ is a tree of rational curves, then $\hh1.C.\O C. = 0$ by~\labelcref{421}.
This shows that \labelcref{C ratl tree} $\gdw$ \labelcref{O_C van}.

For ``\labelcref{C ratl tree} $\gdw$ \labelcref{Q_C van}'', one argues similarly, using instead the sequence
\[ 0 \lto \Q_C \lto \nu_* \Q_{C^\nu} \lto \bigoplus_{P \in \Sing C} \underline \Q_P \lto 0, \]
where $\Q_C$ denotes the constant sheaf on $C$ with values in \Q.
\end{proof}

\PreprintAndPublication{
\begin{prp}[Exceptional trees of rational curves] \label{ratl hom sphere}
Let $\germ0X$ be a normal surface singularity.
The following are equivalent.
\begin{enumerate}
\item \label{exists ratl tree} There exists a log resolution $f \from Y \to X$ such that $E$ is a tree of rational curves.
\item \label{forall ratl tree} For any log resolution $f \from Y \to X$, $E$ is a tree of rational curves.
\item \label{link} The link $L$ of $\germ0X$ is a rational homology sphere.
\end{enumerate}
\end{prp}

Recall that a $3$-manifold $M$ is said to be a rational homology sphere if
\[ \Hh i.M.\Q. \isom \Hh i.S^3.\Q. =
\begin{cases}
  \Q, & i = 0 \text{ or } 3, \\
  0,  & \text{otherwise}
\end{cases} \]
for all $0 \le i \le 3$.

\begin{proof}[Proof of \cref{ratl hom sphere}]
Let $f \from Y \to X$ be any log resolution, with $E = f\inv(0)$.
We have a natural continuous map $L \to E$.
By~\cite[p.~235]{Mumford61}, the induced map $\Hh1.L.\Z. \to \Hh1.E.\Z.$ is surjective with finite kernel.
Hence $\Hh1.L.\Q. \to \Hh1.E.\Q.$ is an isomorphism.
By \cref{ratl tree}, it follows that $E$ is a tree of rational curves if and only if $\Hh1.L.\Q. = 0$.
Since in any case $L$ is a compact connected orientable $3$-manifold, $\Hh1.L.\Q. = 0$ in turn is equivalent to $L$ being a rational homology sphere.
\end{proof}
}
{}

\begin{proof}[Proof of \cref{lowgenus lz}]
Let $\set{ v_1, v_2 }$ be a basis of $\T X$, i.e.~$v_1, v_2 \in \HH0.X.\T X.$ give an isomorphism $\O X^{\oplus 2} \bij \T X$.
Let $\set{ \alpha_1, \alpha_2 }$ be the dual basis of $\Omegar X1$, defined by $\alpha_i(v_j) = \delta_{ij}$.
Furthermore, we may and will assume that $f \from Y \to X$ is the functorial resolution.

\begin{clm} \label{dim1}
We have $\dim \factor{\can X}{f_* \can Y(E)} \le 1$.
\end{clm}

\begin{proof}
Consider the short exact sequence
\[ 0 \lto \underbrace{\factor{f_* \can Y(E)}{f_* \can Y}}_{\ndef \sK} \lto \factor{\can X}{f_* \can Y} \lto \factor{\can X}{f_* \can Y(E)} \lto 0. \]
By~\cite[Prop.~4.45(6)]{KM98}, we have $\dim \factor{\can X}{f_* \can Y} = \pg X0$.
Hence in case~\labelcref{genus1}, we are done.
In case~\labelcref{genus2}, it suffices to show $\sK \ne 0$.
For this, consider the residue sequence
\[ 0 \lto \can Y \lto \can Y(E) \lto \can E \lto 0. \]
Since $\RR1.f.\can Y. = 0$ by Grauert--Riemenschneider vanishing~\cite[Thm.~2.20.1]{Kol07}, $\sK = \HH0.E.\can E.$.
This is Serre dual to $\HH1.E.\O E.$, since $E$ is Cohen--Macaulay.
By \cref{ratl tree}, the latter space is nonzero.
\end{proof}

By \cref{dim1}, the images of $\d\alpha_1$ and $\d\alpha_2$ in $\factor{\can X}{f_* \can Y(E)}$ are linearly dependent.
Possibly interchanging $\alpha_1$ and $\alpha_2$, we may assume that there is a relation
\[ \d\alpha_1 + \lambda \cdot \d\alpha_2 = 0 \in \factor{\can X}{f_* \can Y(E)} \]
for some $\lambda \in \C$.
This means that $\d(\alpha_1 + \lambda \alpha_2)$ extends to $Y$ with logarithmic poles.
Then by \cref{SvS}, $\alpha_1 + \lambda \alpha_2$ extends to $Y$ without poles.
Setting $v_2' \defn -\lambda v_1 + v_2$, the basis $\set{ v_1, v_2' }$ of $\T X$ has as its dual basis $\set{ \alpha_1 + \lambda \alpha_2, \alpha_2 }$.
Replacing $v_2$ by $v_2'$, we may assume the following.

\begin{awlog}
The reflexive $1$-form $\alpha_1$ extends to $Y$ without poles.
\end{awlog}

Now $v_i$ and $\alpha_1$ extend to $\wt v_i \in \HH0.Y.\T Y.$ and $\wt\alpha_1 \in \HH0.Y.\Omegap Y1.$, respectively.
We have that $\wt\alpha_1(\wt v_1) = 1$ on $Y \setminus E$, hence this holds on all of $Y$.
It follows that the vector field $\wt v_1$ does not have any zeros, since $\set{ \wt v_1 = 0 } \subset \big\{ \wt\alpha_1(\wt v_1) = 0 \big\} = \emptyset$.

But $\wt v_1 \in \HHbig 0.Y.\T Y(-\log E).$, that is, $\wt v_1$ is tangent to each irreducible component $E_i \subset E$ at every point of $E_i$.
In particular, $\wt v_1$ vanishes at the singular points of $E$.
It follows that $E$ is a smooth irreducible curve (or empty, in which case $\germ0X$ is smooth and we are done).
Furthermore it carries the nowhere vanishing vector field $\wt v_1|_E$, i.e.~$E$ is an elliptic curve.
Writing down the discrepancy formula
\[ K_Y = f^* K_X + a(E, X) \cdot E \]
and intersecting with $E$, we get $0 = (K_Y + E) \cdot E = (a(E, X) + 1) \cdot E^2$.
Hence $a(E, X) = -1$, as $E^2 < 0$, and thus $\germ0X$ is a log canonical singularity.
From here, there are several ways to conclude that $\germ0X$ is in fact smooth:

\begin{itemize}
\item For log canonical singularities, the Lipman--Zariski conjecture is known by~\cite[Thm.~1.1]{Dru13} or by~\cite[Cor.~1.3]{GK13}.
Note that although~\cite{GK13} is formulated in the algebraic setting, the proofs work verbatim for complex spaces.
Alternatively, we may appeal to Artin approximation~\cite[Thm.~3.8]{ArtinApproximation} in order to see that every normal surface singularity, being isolated, is in fact algebraic.
\item In the surface case, the above result is essentially contained in~\cite{SvS85}:
We have $\factor{\can X}{f_* \can Y(E)} = 0$ by the definition of log canonical singularities, and then \cref{SvS} tells us that all reflexive $1$-forms on $X$ extend to $Y$.
Now one may argue as in~\cite[(1.6)]{SvS85}.
\PreprintAndPublication{
\item Alternatively, there is also a completely elementary argument, which we give in \cref{elliptic lz} below.}
{\item Alternatively, one may also prove the following very special case of the Lipman--Zariski conjecture by a completely elementary argument:
\emph{Let $\germ0X$ be an $n$-dimensional normal isolated log canonical singularity such that for the functorial resolution $f \from Y \to X$, the exceptional locus is irreducible.
Then the Lipman--Zariski conjecture holds for $\germ0X$.}
Details are contained in the preprint version of this article, available at \href{https://arxiv.org/abs/1801.05753}{arxiv:1801.05753 [math.AG]}.
}
\end{itemize}

\noindent
Using either of these arguments, the proof of \cref{lowgenus lz} is finished.
\end{proof}

\PreprintAndPublication{
\begin{prp}[Elementary case of the LZ conjecture] \label{elliptic lz}
Let $\germ0X$ be an $n$-dimensional normal isolated log canonical singularity such that for the functorial resolution $f \from Y \to X$, the exceptional locus is irreducible.
Then the Lipman--Zariski conjecture holds for $\germ0X$.
\end{prp}

\begin{proof}
Let $E \subset Y$ be the exceptional locus of $f$, a smooth projective variety.
We make a case distinction according to whether the tangent sheaf $\T E$ is globally generated or not.

\emph{Case 1: $\T E$ is not globally generated.}
Let $\sF \subsetneq \T E$ be the subsheaf generated by $\HH0.E.\T E.$.
The restriction map $\rho \from \T Y(-\log E) \to \T E$ is surjective, hence $\sG \defn \rho\inv(\sF) \subsetneq \T Y(-\log E)$ also is a proper subsheaf.
By construction, the pullback map of vector fields $\T X \to f_* \T Y(-\log E)$ factors via $f_* \sG$.
By adjointness, also $f^* \from f^* \T X \to \T Y(-\log E)$ factors via $\sG$.
Since $\sG$ is a proper subsheaf, this shows that $f^*$ is not surjective.
As $\T X \isom \O X^{\oplus n}$ is free, taking determinants we obtain a map
\begin{equation} \label{det f^*}
\det(f^*) \from \O Y \to \O Y \big( \! -(K_Y + E) \big)
\end{equation}
which is likewise non-surjective (hence zero) along $E$.
It therefore factors via a map
\[ \O Y \lto \O Y \big( \! -(K_Y + 2 E) \big), \]
which is furthermore isomorphic outside of $E$.
This immediately implies that the discrepancy $a(E, X) \le -2$, contradicting the assumption that $\germ0X$ is log canonical.

\emph{Case 2: $\T E$ is globally generated.}
The existence of the map $\det(f^*)$ from \labelcref{det f^*} shows that $a(E, X) \le -1$, which implies $a(E, X) = -1$ as $\germ0X$ is assumed to be log canonical.
Then $K_E = (K_Y + E)|_E = (f^* K_X)|_E = 0$ and by \cref{gg det=0} below, $\T E \isom \O E^{\oplus (n - 1)}$ is trivial.
Consider now the residue sequence for $E \subset Y$ and its restriction to $E$,
\begin{equation} \label{517}
\xymatrix{
0 \ar[r] & \Omegap Y1 \ar[r] \ar[d] & \Omegap Y1(\log E) \ar[r] \ar[d] & \O E \ar[r] \ar@{=}[d] & 0 \\
0 \ar[r] & \Omegap E1 \ar[r] & \Omegap Y1(\log E)\big|_E \ar[r] & \O E \ar[r] & 0.
}
\end{equation}
By~\cite[Lemma~3.5]{GK13}, the extension class of the first line of \labelcref{517} is $\cc1{\O Y(E)} \in \HH1.Y.\Omegap Y1.$.
The extension class of the second line is then $\cc1{\O E(E)} \in \HH1.E.\Omegap E1.$, which is nonzero by the Negativity Lemma~\cite[Lemma~3.39]{KM98}.
Thus the sequences in \labelcref{517} do not split.
In particular, the dual of the lower-row sequence
\[ \xymatrix{
0 \ar[r] & \O E \ar[r] & \T Y(-\log E)\big|_E \ar^-{\rho_E}[r] & \T E \isom \O E^{\oplus (n - 1)} \ar[r] & 0
} \]
does not split.
It follows that the map of global sections
\[ \HH0.E.\rho_E. \from \HH0.E.\T Y(-\log E)\big|_E. \to \HH0.E.\T E. \]
is not surjective.
The rest of the argument proceeds exactly as in Case 1:
Let $\sF \subsetneq \T E$ be the proper subsheaf generated by the image of $\HH0.E.\rho_E.$, and set $\sG \defn \rho\inv(\sF)$.
The pullback map factorizes as $\T X \to f_* \sG \to f_* \T Y(-\log E)$.
It follows that $a(E, X) \le -2$ and we arrive at a contradiction.
\end{proof}

\begin{lem}[Criterion for triviality] \label{gg det=0}
Let $X$ be a projective variety and $\sE$ a rank $r$ vector bundle on $X$ with trivial determinant, $\det \sE \isom \O X$.
If $\sE$ is globally generated at some point $x \in X$, then $\sE \isom \O X^{\oplus r}$.
\end{lem}

\begin{proof}
Take $r$ sections $s_1, \dots, s_r \in \HH0.X.\sE.$ which generate $\sE$ at $x$, i.e.~the images of the $s_i$ in $\factor{\sE_x}{\frm_x\sE_x}$ form a basis of that vector space.
Then $s_1 \wedge \cdots \wedge s_r \in \HH0.X.\det\sE.$ is nonzero, hence nowhere vanishing.
It follows that the sections $s_i$ generate $\sE$ everywhere.
The map $\O X^{\oplus r} \to \sE$ defined by them is thus an isomorphism.
\end{proof}
}
{}

\section{Proof of \cref{TX gen nef}}

Let $X$ be a projective surface with locally free and generically nef tangent sheaf.
We want to show that $X$ is smooth and classify the possibilities for $X$.
Since the proof of smoothness involves some nested case distinctions, it may be a little hard to follow.
For the reader's convenience, the structure of the argument is therefore depicted in \cref{proof structure}.

\subsection*{Step 1: $X$ is smooth}

\begin{figure}
  \centering
  \tikzstyle{every node} = [draw=black, thick, anchor=west]
  \tikzstyle{endnode} = [draw=green!75!black, fill=black!10]

  \begin{tikzpicture}
    [ grow via three points = {one child at (-1.1, -0.7) and
        two children at (-1.1, -0.7) and (-1.1, -1.4)},
      edge from parent path = { ($(\tikzparentnode.south west)!0.2!(\tikzparentnode.south)$) |- (\tikzchildnode.west) }
    ]
    \node { Consider $\hh0.X.{\RR1.f.\O S.}.$ }
    child {
      node[endnode] { $\cdots = 1$: done by \labelcref{genus1} -- Case \textsf A }
    }
    child {
      node { $\cdots = 2$: Consider $|\Sing X|$ }
      child {
        node { $\cdots = 1$: Consider $\Exc(f)$ }
        child {
          node[endnode] { \dots not a tree of rational curves: done by \labelcref{genus2} -- Case \textsf B }
        }
        child {
          node[endnode] { \dots tree of rational curves: get $-2 = K_S \cdot F = 0$ \Lightning\ -- Case \textsf C }
        }
      }
      child[missing] { }
      child[missing] { }
      child {
        node[endnode] { $\cdots = 2$: done by \labelcref{genus1} -- Case \textsf D }
      }
    };
  \end{tikzpicture}

  \caption{Structure of Step 1 in the proof of \cref{TX gen nef}}
  \label{proof structure}
\end{figure}
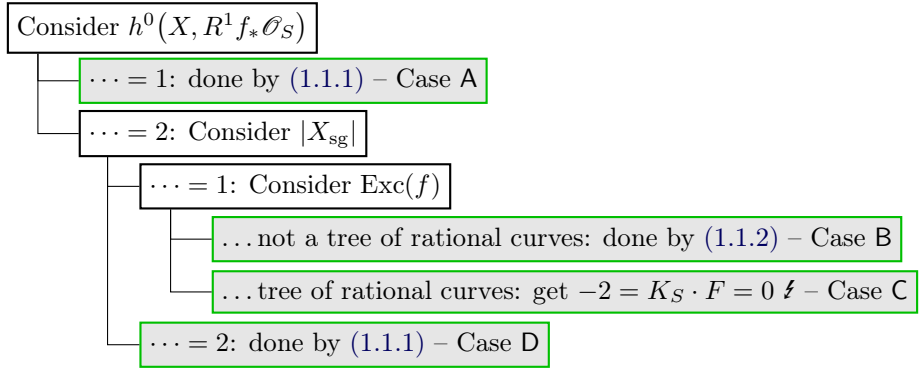

By~\cite[Thm.~3]{Lip65}, $X$ is normal.
We may assume that $X$ is not smooth, otherwise we are done.
Let $f \from S \to X$ be the minimal resolution, i.e.~$K_S$ is $f$-nef (equivalently, $f$ does not contract any $(-1)$-curves).

\begin{clm} \label{kappa S}
The Kodaira dimension $\kappa(S) = -\infty$, and $\kappa(X, K_X) \le 0$.
If $\HH0.X.\can X. \ne 0$, then $\can X \isom \O X$.
\end{clm}

\begin{proof}
Since $\T X$ is generically nef and $\can X = \det(\T X)\dual$, we have $K_X \cdot H \le 0$ for some ample divisor $H$ on $X$.
Consequently, $\kappa(S) \le \kappa(X, K_X) \le 0$.
If $\kappa(X, K_X) = -\infty$, we are done.
If $\kappa(X, K_X) = 0$, then some multiple $mK_X$ is effective and hence $K_X \cdot H \ge 0$.
Combined with the above, we get $K_X \cdot H = 0$.
Consequently, $mK_X$ is trivial and $K_X$ is torsion.
By \cite[Cor.~4.3]{KM98}, we have
\[ K_S = f^* K_X + E \sim_\Q E \]
with $E \le 0$ an anti-effective $f$-exceptional divisor.
If $E = 0$, then the singularities of $X$ are canonical, hence $X$ is smooth~\cite[Thm.~1.1]{Dru13}, \cite[Cor.~1.3]{GK13}.
So, $E \lneq 0$ and we again arrive at $\kappa(S) = \kappa(S, E) = -\infty$.
The last statement is clear from the above.
\end{proof}

\begin{clm} \label{leray}
The Leray spectral sequence associated to $f_* \O S$ yields a five-term sequence
\begin{equation*}
\begin{aligned}
0 & \lto \HH1.X.\O X. \lto \underbrace{\HH1.S.\O S.}_{\dim \, \le \, 1} \lto \underbrace{\HH0.X.{\RR1.f.\O S.}.}_{\imp \, \dim \, \le \, 2} \lto \\
  & \lto \underbrace{\HH2.X.\O X.}_{\dim \, \le \, 1} \lto \underbrace{\HH2.S.\O S.}_{= \, 0},
\end{aligned}
\end{equation*}
where the dimensions are as shown.
In particular, $\hh0.X.{\RR1.f.\O S.}. \le 2$.
\end{clm}

\begin{proof}
By \cref{kappa S} and Serre duality on both $S$ and $X$, we have $\hh2.S.\O S. = \hh0.S.\can S. = 0$ and $\hh2.X.\O X. = \hh0.X.\can X. \le 1$.
It remains to show that $\hh1.S.\O S. \le 1$.
The surface $S$ is a blowup of either $\PP2$ or a ruled surface over a curve $C$, say of genus $g$.
In the first case, $\HH1.S.\O S. = 0$ and we are done.
In the second case, let $\pi \from S \to C$ be the natural map, and pick a general sufficiently ample divisor $H$ on $X$.
Let $H_S$ be its strict transform on $S$.
Since $H$ misses the singular points of $X$, we see that $H_S \isom H$ and that $\T S\big|_{H_S} \isom \T X\big|_H$ is nef.
The differential of $\pi$ restricted to $H_S$,
\[ (\d\pi)\big|_{H_S} \from \T S\big|_{H_S} \lto (\pi^*\T C)\big|_{H_S}, \]
shows that $(\pi^*\T C)\big|_{H_S} = (\pi|_{H_S})^*\T C$ contains a line bundle of non-negative degree, hence has non-negative degree itself.
Consequently, $\deg \T C = 2 - 2g \ge 0$, which implies $g \le 1$.
But $\hh1.S.\O S. = \hh1.C.\O C. = g$.
This finishes the proof.
\end{proof}

Since the singularities of $X$ are isolated, the following formula holds:
\[ \hh0.X.{\RR1.f.\O S.}. = \sum_{x \in \Sing X} \pg Xx. \]
Furthermore, rational singularities cannot occur by \labelcref{genus1}.
This leaves us with the following possibilities (cf.~\cref{proof structure}):
\begin{itemize}
\item Case \textsf A: $X$ has exactly one singular point, which is of genus one.
\item Case \textsf B: $X$ has exactly one singular point, which is of genus two. $\Exc(f)$ is not a tree of rational curves.
\item Case \textsf C: Same as Case \textsf B, but $\Exc(f)$ \emph{is} a tree of rational curves.
\item Case \textsf D: $X$ has exactly two singular points. Both of them are of genus one.
\end{itemize}
In Cases \textsf A and \textsf D, \labelcref{genus1} implies that $X$ is smooth, and we are done.
In Case \textsf B we use \labelcref{genus2} instead.
Case \textsf C will be excluded by a more careful analysis.

Indeed, note that by \cref{leray}, it can happen that $\hh0.X.{\RR1.f.\O S.}. = 2$ only if both $\HH1.S.\O S.$ and $\HH2.X.\O X.$ are one-dimensional.
The non-vanishing of $\HH1.S.\O S.$ means that $S$ is a blowup of a ruled surface over an elliptic curve $C$, hence comes equipped with a natural map $\pi \from S \to C$.
On the other hand, by \cref{kappa S} the non-vanishing $\HH2.X.\O X. \ne 0$ implies $\can X \isom \O X$, hence $K_S$ has a representative (not necessarily effective) whose support is contained in $\Exc(f)$.
But $\pi \big( \!\Exc(f) \big)$ is a point, since every component of $\Exc(f)$ is a rational curve, while $C$ is elliptic.
This clearly implies that $K_S \cdot F = 0$, where $F \subset S$ is a general fibre of $\pi$.
On the other hand, $F \isom \PP1$ and $F^2 = 0$, so $K_S \cdot F = -2$ by adjunction.
We arrive at a contradiction, showing that Case \textsf C in fact cannot occur.
This finishes the proof of the first part of \cref{TX gen nef}, namely that $X$ is smooth.

\subsection*{Step 2: Classification}

It remains to classify all smooth projective surfaces $X$ with $\T X$ generically nef.
To this end, let $X$ be such a surface and let $f \from X \to X_0$ be a minimal model, i.e.~$X_0$ does not contain any $(-1)$-curves.
Since $\T X$ is generically nef, we have $\kappa(X) \le 0$.
\begin{itemize}
\item If $\kappa(X) = -\infty$, then either $X_0 \isom \PP2$ or $\pi_0 \from X_0 \to C$ is a ruled surface over a curve $C$ of genus $g$.
By the argument in the proof of \cref{leray}, it follows that $g \le 1$.
If $X_0 \isom \PP2$, then either $f$ is an isomorphism and we are in Case~\labelcref{P2}, or $f$ is not an isomorphism and Case~\labelcref{ratl ruled} occurs.
If $X_0$ is ruled and $g = 0$, we are likewise in Case~\labelcref{ratl ruled}.
If $g = 1$ and $\pi \from X \to X_0 \to C$ has a non-reduced fibre, let $H \subset X$ be a general sufficiently ample divisor.
In every point where $H$ meets a non-reduced component of a fibre of $\pi$, the map
\[ (\d\pi)\big|_H \from \T X\big|_H \lto (\pi^*\T C)\big|_H \isom \O H \]
is not surjective.
This shows that $\T X\big|_H$ is not nef, as it has a line bundle quotient of negative degree.
Consequently, all fibres of $\pi$ are reduced and we are in Case~\labelcref{ruled ell}.
\item If $\kappa(X) = 0$, we claim that $X = X_0$ is already minimal.
Otherwise, as $K_{X_0} \sim_\Q 0$, the canonical divisor of $X$ would be effective and nonzero.
Then $K_X \cdot H > 0$ for any $H$ ample on $X$, contradicting the generic nefness of $\T X$.
By the Kodaira--Enriques classification~\cite[Table~10 on p.~244]{BHPV04}, this accounts for Cases~\labelcref{abelian} and~\labelcref{K3}.
\end{itemize}

Conversely, we need to show that the above surfaces enjoy the positivity properties claimed in \cref{TX gen nef}.
\begin{itemize}
\item Case~\labelcref{P2}: If $X \isom \PP2$, the tangent bundle is ample by the Euler sequence.

\item Case~\labelcref{ratl ruled}: Let $\pi_0 \from X_0 \to \PP1$ be a rational ruled surface such that $X$ is a blowup of $X_0$.
Let $C_0 \subset X_0$ be a section of $\pi_0$ with $C_0^2 = -n \le 0$, and let $F \subset X_0$ be a fibre of $\pi_0$.
Consider the relative tangent sheaf sequence of $\pi \from X \to X_0 \to \PP1$,
\[ 0 \lto \T{X/\PP1} \lto \T X \lto \cQ \lto 0, \]
where $\cQ \subset \pi^* \T{\PP1}$ is the image of $\d\pi$.
Since a general sufficiently ample $H \subset X$ misses the finitely many singular points of the torsion-free sheaf $\T{X/\PP1}$, restricting to $H$ preserves injectivity:
\begin{equation} \label{extension}
0 \lto \T{X/\PP1}\big|_H \lto \T X\big|_H \lto \cQ\big|_H \lto 0.
\end{equation}
We will show that for a suitable choice of $H$, both the kernel and the cokernel in \labelcref{extension} are ample line bundles.

Clearly $\T{X/\PP1}$ and $f^* \T{X_0/\PP1}$ agree outside the $f$-exceptional set.
Also the line bundle $\T{X_0/\PP1}$ is easily calculated to equal $2 \ C_0 + n F$, in divisor notation.
Hence $\T{X/\PP1}$ equals $f^*(2 \ C_0 + n F) + E$, for some (not necessarily effective) $f$-exceptional divisor $E$ on $X$.
Let $A$ be ample on $X_0$, and set $H_\ell \defn H + \ell \cdot f^* A$ for $\ell \ge 0$.
Then
\[ \cc1{\T{X/\PP1}} \cdot H_\ell = \underbrace{\deg \big( \T{X/\PP1}\big|_H \big)}_{\text{indep.~of $\ell$}} + \ \ell \cdot \underbrace{(2 \ C_0 + n F) \cdot A}_{>0}, \]
which is positive for $\ell \gg 0$.
We see that up to replacing $H$ by $H_\ell$, the line bundle $\T{X/\PP1}\big|_H$ is ample.

For $\cQ\big|_H$, the argument is similar.
The sheaf $\cQ$ agrees with $\pi^* \T{\PP1}$ outside the $f$-exceptional set, so it equals $f^*(2 \ F) + E'$ for some $f$-exceptional divisor $E'$.
Therefore
\[ \cc1{\cQ} \cdot H_\ell = \underbrace{\deg \big( \cQ\big|_H \big)}_{\text{indep.~of $\ell$}} + \ \ell \cdot \underbrace{(2 \ F) \cdot A}_{>0} \]
is positive for $\ell \gg 0$.

Picking $\ell$ sufficiently large for both of the above arguments to work, the sheaf $\T X\big|_H$ is exhibited by \labelcref{extension} as an extension of ample bundles.
Thus it is ample itself~\cite[Prop.~6.1.13]{Laz04b}.

\item Case~\labelcref{ruled ell}: This case is completely analogous to the previous one, hence we only give a sketch of the proof.
We keep the same notation as before.
The analogue of \labelcref{extension} reads
\[ 0 \lto \T{X/C}\big|_H \lto \T X\big|_H \lto \underbrace{(\pi^* \T C)\big|_H}_{\isom \O H} \lto 0. \]
To justify surjectivity on the right-hand side, note that by assumption the fibres of $\pi$ are reduced and so there are only finitely many $x \in X$ which are singular points of $\pi\inv \big( \pi(x) \big)$.
These are exactly the points where $\d\pi$ fails to be surjective, and they are missed by the general ample divisor $H$.

Since $\T{X_0/C}$ still equals $2 \ C_0 + n F$, exactly the same calculation as above shows that $\T{X/C}\big|_H$ has positive degree, for suitable choice of $H$.
It follows that $\T X\big|_H$ is nef, being an extension of nef bundles.

\item Case~\labelcref{abelian}: The tangent bundle of an abelian surface is trivial, in particular it is nef.
A bi-elliptic surface $X$ admits a finite \'etale map $\gamma \from E_1 \x E_2 \to X$ from a product of elliptic curves $E_1, E_2$.
The pullback $\gamma^* \T X \isom \T{E_1 \x E_2}$ is trivial, in particular nef.
Then also $\T X$ itself is nef~\cite[Prop.~6.1.8]{Laz04b}.

\item Case~\labelcref{K3}: Let $H \subset X$ be a general sufficiently ample divisor.
By~\cite[Cor.~6.4]{Miy87b}, the restriction of the cotangent bundle $\Omegap X1\big|_H$ is nef.
As $\cc1X = 0$, also its dual $\T X\big|_H$ is nef.
\end{itemize}

\noindent
This finishes the proof of \cref{TX gen nef}. \qed

\section{Proof of Corollaries \labelcref{X almost hom} and \labelcref{TX trivial}}

\subsection*{Proof of \cref{X almost hom}}

Assume that $X$ is almost homogeneous, and let $f \from S \to X$ be the functorial resolution.
Then also $S$ is almost homogeneous, as $\HH0.S.\T S. = \HH0.X.\T X.$.
According to~\cite[Main Theorem]{PottersAlmostHomogeneous}, the surface $S$ is one of the following:

\stepcounter{thm}
\begin{enumerate}
\item\label{pot ratl} the projective plane $\PP2$, a rational ruled surface, or a blowup thereof,
\item\label{pot ruled ell} a projective bundle $\P_C(\sE) \to C$ over an elliptic curve $C$, where the vector bundle $\sE$ either
  \begin{itemize}
  \item decomposes as $\sL \oplus \O C$ for some $\sL \in \Picn(C)$, or
  \item it is the unique non-trivial extension $0 \to \O C \to \sE \to \O C \to 0$,
  \end{itemize}
\item\label{pot Hopf} an abelian Hopf surface, i.e.~a surface with universal covering $\C^2 \setminus \set 0$ and abelian fundamental group,
\item\label{pot torus} a complex $2$-torus.
\end{enumerate}

\begin{clm} \label{can X}
We have $\hh2.X.\O X. \le 1$, and equality holds if and only if $\can X \isom \O X$.
\end{clm}

\begin{proof}
By Serre duality, $\hh2.X.\O X. = \hh0.X.\can X.$.
Since $\T X$ is globally generated at some point, its determinant $\can X \dual$ has a nonzero section.
Therefore, as soon as $\hh0.X.\can X. \ne 0$, it follows that $\can X \isom \O X$.
\end{proof}

\noindent
We treat each case in Potters' list separately.

\begin{itemize}
\item In case~\labelcref{pot ratl}, the Leray spectral sequence associated to $f_* \O S$ (see \cref{leray}) yields $\hh0.X.{\RR1.f.\O S.}. \le 1$, using \cref{can X}.
Arguing as in the proof of \cref{TX gen nef}, we conclude by \labelcref{genus1} that $X$ is smooth.
\item In case~\labelcref{pot ruled ell}, we get $\hh0.X.{\RR1.f.\O S.}. \le 2$.
We can still argue as in Case \textsf C of the proof of \cref{TX gen nef} to obtain smoothness of $X$.
\item In case~\labelcref{pot Hopf}, note that $S$ does not contain any negative curves, as $b_2(S) = 0$ \cite[Thm.~18.4]{BHPV04}.
Hence $f$ is an isomorphism and $X = S$ is smooth.
\item In case~\labelcref{pot torus}, $S$ does not contain any negative curves since it is homogeneous.
Again, $f$ is an isomorphism and $X$ is smooth.
\end{itemize}

\noindent
This ends the proof of \cref{X almost hom}. \qed

\subsection*{Proof of \cref{TX trivial}}

Let $\set{ v_1, v_2 }$ be a global basis of $\T X$.
Either $v_1$ and $v_2$ commute, i.e.~the Lie bracket $[v_1, v_2] = 0$, or (after a suitable change of basis) we may assume $[v_1, v_2] = v_2$.
In either case, for the dual basis $\set{ \alpha_1, \alpha_2 }$ of $\Omegar X1$ we get $\d\alpha_1 = 0$.
Hence $\alpha_1$ extends to $\wt\alpha_1 \in \HH0.Y.\Omegap Y1.$, for a resolution $Y \to X$.
We conclude as in the proof of \cref{lowgenus lz} that $X$ is smooth.

If we were in the case $[v_1, v_2] = v_2$, then $\d\alpha_2$ would be a nowhere-vanishing $2$-form on $X$.
Since any $1$-form on a smooth compact complex surface is closed, this case actually cannot occur and so $v_1$ and $v_2$ commute.
The flow maps associated to these vector fields then show that $X$ is a complex torus. \qed

\section{Two examples of surface singularities} \label{examples}

\begin{exm} \label{pg=1 not lc}
We give an example of a normal Gorenstein surface singularity $\germ0X$ of genus $\pg X0 = 1$ which is not log canonical.
Consider a star-shaped tree of five smooth rational curves $C_0 + \dots + C_4$ in a smooth surface $Y$, having the following intersection matrix (empty entries are zero):
\begin{equation} \label{matrix 1}
\begin{pmatrix}
-2 &  1 &  1 &  1 &  1 \\
 1 & -3 &    &    &    \\
 1 &    & -3 &    &    \\
 1 &    &    & -3 &    \\
 1 &    &    &    & -3
\end{pmatrix}
\end{equation}
Such a configuration clearly exists, e.g.~by starting with $C_0$ the zero section of the line bundle $\O{\PP1}(2)$, blowing up four distinct points on $C_0$, and then blowing up two more points on each of the exceptional curves.
A short calculation shows that \labelcref{matrix 1} is negative definite, hence the curves $C_0 + \cdots + C_4 \subset Y$ can be blown down to the desired normal singularity $\germ0X$ by~\cite[p.~367]{Grauert62}.

The fundamental cycle $Z$ of this singularity (i.e.~the unique minimal nonzero effective relatively anti-nef exceptional divisor) is easily seen to be $Z = 2 \ C_0 + C_1 + \cdots + C_4$.
It satisfies $\chi(Z) = 0$, while $\chi(Z') = 1 > 0$ for $0 < Z' < Z$.
Hence $\germ0X$ is minimally elliptic in the sense of~\cite[Def.~3.2]{Laufer77}.
Then by~\cite[Thm.~3.10]{Laufer77}, the genus $\pg X0 = 1$ and $\germ0X$ is Gorenstein.
Furthermore, by~\cite[Thm.~3.4(2)]{Laufer77} we have $K_Y = -Z$.
In particular, the discrepancy $a(X, C_0) = -2$ and so $\germ0X$ is not log canonical.
\end{exm}

\begin{exm} \label{ratl tree pg large}
Fix any integer $g \ge 0$.
We give an example of a normal surface singularity $\germ0X$ which is obtained by contracting a tree of rational curves and with genus $\pg X0 \ge g$.
Consider a tree of smooth rational curves $C_0 + \dots + C_{g+3} \subset Y$, having the following intersection matrix for some $d \ge 1$:
\[ A_{g,d} =
\begin{pmatrix}
    -2 &  1 &  1 & \cdots &  1 \\
     1 & -d &    &        &    \\
     1 &    & -d &        &    \\
\vdots &    &    & \ddots &    \\
     1 &    &    &        & -d
\end{pmatrix}
\in \R^{(g+4) \x (g+4)}
\]
This can be constructed in a similar way as in the previous example.
We claim that for $d \ge g + 3$, the matrix $A_{g,d}$ is negative definite.
\PreprintAndPublication{
To this end, we will use the criterion given by \cref{pos def crit} below.}
{To this end, we will use the following well-known criterion:
\emph{Let $A$ be a real symmetric $n \x n$-matrix with non-negative off-diagonal entries.
Then $A$ is negative definite if and only if there exists a vector $v \in (\R^+)^n$ such that $Av \in (\R^-)^n$.}
}%
Consider $v = (g + 2, 1, \dots, 1) \in \R^{g+4}$.
Then $A_{g,d} \cdot v = (-g - 1, g + 2 - d, \dots, g + 2 - d)$.
For $d \ge g + 3$, this has only negative entries and we may apply the criterion.

Contracting $C_0 + \cdots + C_{g+3}$ yields a singularity $\germ0X$.
For $Z = 2 \ C_0 + C_1 + \cdots + C_{g+3}$, we have a short exact sequence
\[ 0 \lto \underbrace{\O{C_0}(-\Red Z)}_{\isom \O{\PP1}(-g-1)} \lto \O Z \lto \O{\Red Z} \lto 0. \]
So $\euler Z.\O Z. = \euler \PP1.\O{\PP1}(-g-1). + \euler \Red Z.\O{\Red Z}. = -g + 1$ and $\HH0.Z.\O Z. = \C$.
It follows that $\dim \HH1.Z.\O Z. = g$ and hence $\pg X0 \ge g$, as desired.
\end{exm}

\PreprintAndPublication{
The following matrix-theoretic result is used e.g.~in~\cite{Grauert62} without a reference.
For the reader's convenience, we provide a proof.

\begin{prp}[Criterion for positive definiteness] \label{pos def crit}
Let $A = (a_{ij})$ be a real symmetric $n \x n$-matrix with non-positive off-diagonal entries.
The following are equivalent:
\begin{enumerate}
\item\label{1072} The matrix $A$ is positive definite.
\item\label{1073} There exists a vector $v \in (\R^+)^n$ such that also $Av \in (\R^+)^n$.
\end{enumerate}
\end{prp}

\begin{proof}
``\labelcref{1072} $\imp$ \labelcref{1073}'': Let $A = LL^{\mathsf T}$ be the Cholesky decomposition of $A$, where $L$ is lower triangular and has positive diagonal entries.
It is easy to see that $L$ has non-positive off-diagonal entries since $A$ does.
Thus $L\inv$ has all entries non-negative, and then the same is true of $A\inv = L^{-\mathsf T} L\inv$.
It follows that $v = A\inv \cdot (1, \dots, 1)$ has the desired properties.

``\labelcref{1073} $\imp$ \labelcref{1072}'': It suffices to show that $\det(A) > 0$, since we can run the same argument on the leading principal minors of $A$ and then apply Sylvester's criterion.
Write $v = (v_1, \dots, v_n)$.
Replacing $A$ by $A \cdot \mathrm{diag}(v_1, \dots, v_n)$, we may assume that $v = (1, \dots, 1)$.
Note that here we lose symmetry of $A$, but this is not a problem.
The condition $Av \in (\R^+)^n$, combined with $a_{ij} \le 0$ for $i \ne j$, easily implies that $A$ is strictly diagonally dominant and that $a_{ii} > 0$ for all $i$.
Let $D = \mathrm{diag}(a_{11}, \dots, a_{nn})$ be the diagonal matrix containing the diagonal entries of $A$.
Connect $A$ and $D$ by the line segment $A(t) = (1 - t)A + tD$ for $0 \le t \le 1$.
Obviously, all the $A(t)$ are strictly diagonally dominant, hence invertible.
As $\det A(1) = \det D = a_{11} \cdots a_{nn} > 0$ and $\det A(t)$ is continuous in $t$, it follows that $\det A = \det A(0) > 0$, as desired.
\end{proof}
}
{}

\providecommand{\bysame}{\leavevmode\hbox to3em{\hrulefill}\thinspace}
\providecommand{\MR}{\relax\ifhmode\unskip\space\fi MR}
\providecommand{\MRhref}[2]{%
  \href{http://www.ams.org/mathscinet-getitem?mr=#1}{#2}
}
\providecommand{\href}[2]{#2}


\begin{thebibliography}{BHPV04}

\bibitem[Art69]{ArtinApproximation}
{\sc M.~Artin}: \emph{Algebraic approximation of structures over complete local
  rings}, Inst. Hautes \'{E}tudes Sci.~Publ.~Math. (1969), no.~36, 23--58.

\bibitem[Bal06]{Ballico06}
{\sc E.~Ballico}: \emph{Compact normal complex surfaces with locally free and
  non-negative tangent sheaf}, Int. J. Pure Appl. Math. \textbf{29} (2006),
  no.~3, 311--316.

\bibitem[BHPV04]{BHPV04}
{\sc W.~P. Barth, K.~Hulek, C.~A.~M. Peters, and A.~Van~de Ven}: \emph{Compact
  complex surfaces}, second ed., Ergebnisse der Mathematik und ihrer
  Grenzgebiete. 3. Folge. A Series of Modern Surveys in Mathematics [Results in
  Mathematics and Related Areas. 3rd Series. A Series of Modern Surveys in
  Mathematics], vol.~4, Springer-Verlag, Berlin, 2004.

\bibitem[Bec78]{Becker78}
{\sc J.~Becker}: \emph{Higher derivations and integral closure}, Amer.~J.~Math.
  \textbf{100} (1978), no.~3, 495--521.

\bibitem[BGK14]{BiswasGurjarKolte14}
{\sc I.~Biswas, R.~V. Gurjar, and S.~U. Kolte}: \emph{On the {Z}ariski-{L}ipman
  conjecture for normal algebraic surfaces}, J. Lond. Math. Soc. \textbf{90}
  (2014), no.~1, 270--286.

\bibitem[Dru14]{Dru13}
{\sc S.~Druel}: \emph{{The Zariski--Lipman conjecture for log canonical
  spaces}}, Bull. London Math. Soc. \textbf{46} (2014), no.~4, 827--835.

\bibitem[Fle88]{Flenner88}
{\sc H.~Flenner}: \emph{{Extendability of differential forms on non-isolated
  singularities}}, Invent.~Math. \textbf{94} (1988), no.~2, 317--326.

\bibitem[GK14]{GK13}
{\sc P.~Graf and S.~J. Kov{\'a}cs}: \emph{{An optimal extension theorem for
  $1$-forms and the {L}ipman-{Z}ariski Conjecture}}, Documenta Math.
  \textbf{19} (2014), 815--830.

\bibitem[Gra62]{Grauert62}
{\sc H.~Grauert}: \emph{{{\"U}ber Modifikationen und exzeptionelle analytische
  Mengen}}, Math. Annalen \textbf{146} (1962), 331--368.

\bibitem[Har77]{Har77}
{\sc R.~Hartshorne}: \emph{Algebraic geometry}, Graduate Texts in Mathematics,
  vol.~52, Springer-Verlag, New York, 1977.

\bibitem[Kol07]{Kol07}
{\sc J.~Koll{\'a}r}: \emph{{Lectures on resolution of singularities}}, Annals
  of Mathematics Studies, vol. 166, Princeton University Press, Princeton, NJ,
  2007.

\bibitem[KM98]{KM98}
{\sc J.~Koll{\'a}r and S.~Mori}: \emph{Birational geometry of algebraic
  varieties}, Cambridge Tracts in Mathematics, vol. 134, Cambridge University
  Press, Cambridge, 1998.

\bibitem[Lau77]{Laufer77}
{\sc H.~B. Laufer}: \emph{On minimally elliptic singularities}, Amer.~J.~Math.
  \textbf{99} (1977), no.~6, 1257--1295.

\bibitem[Laz04]{Laz04b}
{\sc R.~Lazarsfeld}: \emph{{Positivity in Algebraic Geometry II}}, Ergebnisse
  der Mathematik und ihrer Grenzgebiete, 3. Folge, vol.~49, Springer-Verlag,
  Berlin, 2004.

\bibitem[Lip65]{Lip65}
{\sc J.~Lipman}: \emph{{Free Derivation Modules on Algebraic Varieties}},
  Amer.~J.~Math. \textbf{87} (1965), no.~4, 874--898.

\bibitem[MS87]{MehtaSrinivas87}
{\sc V.~B. Mehta and V.~Srinivas}: \emph{Varieties in positive characteristic
  with trivial tangent bundle}, Compositio Math. \textbf{64} (1987), no.~2,
  191--212, With an appendix by Srinivas and M. V. Nori.

\bibitem[Miy87]{Miy87b}
{\sc Y.~Miyaoka}: \emph{{The Chern Classes and Kodaira Dimension of a Minimal
  Variety}}, Advanced Studies in Pure Mathematics, Algebraic Geometry, Sendai
  1985, vol.~10, 1987, pp.~449--476.

\bibitem[Mum61]{Mumford61}
{\sc D.~Mumford}: \emph{The topology of normal singularities of an algebraic
  surface and a criterion for simplicity}, Inst.~Hautes {\'E}tudes
  Sci.~Publ.~Math. (1961), no.~9, 5--22.

\bibitem[OR88]{OeljeklausRichthofer88}
{\sc K.~Oeljeklaus and W.~Richthofer}: \emph{Linearization of holomorphic
  vector fields and a characterization of cone singularities}, Abh. Math. Sem.
  Univ. Hamburg \textbf{58} (1988), 63--87.

\bibitem[Pot69]{PottersAlmostHomogeneous}
{\sc J.~Potters}: \emph{On almost homogeneous compact complex analytic
  surfaces}, Invent. Math. \textbf{8} (1969), 244--266.

\bibitem[Ste85]{Ste85}
{\sc J.~H.~M. Steenbrink}: \emph{Vanishing theorems on singular spaces},
  Ast{\'e}risque \textbf{130} (1985), 330--341.

\bibitem[SvS85]{SvS85}
{\sc J.~H.~M. Steenbrink and D.~van Straten}: \emph{Extendability of
  holomorphic differential forms near isolated hypersurface singularities},
  Abh. Math. Sem. Univ. Hamburg \textbf{55} (1985), 97--110.

\bibitem[Wan54]{Wang54}
{\sc H.-C. Wang}: \emph{Complex parallisable manifolds}, Proc. Amer. Math. Soc.
  \textbf{5} (1954), 771--776.

\end{thebibliography}
\end{document}